\documentclass{amsart}

\usepackage[letterpaper,top=2.2cm,bottom=2.2cm,left=2.2cm,right=2.2cm]{geometry}
\usepackage[T1]{fontenc}
\usepackage[utf8]{inputenc}
\usepackage{amsmath,amssymb,amsthm,mathtools}
\usepackage{mathrsfs}
\usepackage{enumitem}
\usepackage{booktabs}
\usepackage{array}
\usepackage{graphicx}
\usepackage{verbatim}
\usepackage{algorithm}
\usepackage{algorithmicx}
\usepackage{algpseudocode}
\usepackage[colorlinks=true,linkcolor=blue,citecolor=blue,urlcolor=blue]{hyperref}

\newtheorem{theo}{Theorem}[section]
\newtheorem{prop}[theo]{Proposition}
\newtheorem{lem}[theo]{Lemma}
\newtheorem{cor}[theo]{Corollary}
\newtheorem{ass}[theo]{Assumption}

\theoremstyle{definition}
\newtheorem{defin}[theo]{Definition}

\theoremstyle{remark}
\newtheorem{rem}[theo]{Remark}

\newcommand{\R}{\mathbb{R}}
\newcommand{\Z}{\mathbb{Z}}
\newcommand{\N}{\mathbb{N}}
\newcommand{\E}{\mathbb{E}}
\newcommand{\PP}{\mathbb{P}}

\newcommand{\ind}{\mathbf{1}}
\newcommand{\dd}{\,\mathrm{d}}

\algnewcommand{\LeftComment}[1]{\Statex \(\triangleright\) #1}

\newcommand{\safeincludegraphics}[2][]{
  \IfFileExists{#2}{
    \includegraphics[#1]{#2}
  }{
    \fbox{
      \begin{minipage}[c][4cm][c]{0.85\linewidth}
      \centering
      Missing figure file:\\[0.2cm]
      \texttt{#2}
      \end{minipage}
    }
  }
}

\title[Perfect simulation for reset Hawkes processes]{Perfect simulation for interacting Hawkes processes with reset-induced variable length memory}

\author{Branda P.\,I. Goncalves and Lucien Mauffret}

\address{B.\,P.\,I. Goncalves: Universit\'e Paris Est Cr\'eteil, 61 avenue du G\'en\'eral de Gaulle, Cr\'eteil, France}
\email{branda.goncalves@u-pec.fr}

\address{L. Mauffret: Universit\'e Paris 1 Panth\'eon-Sorbonne, 90 rue de Tolbiac, 75013 Paris, France}
\email{lucien.mauffret@univ-paris1.fr}

\date{}

\begin{document}

\begin{abstract}
We study a class of interacting nonlinear Hawkes point processes on the integer lattice in which each component is reset after its own jumps. The intensity of a component depends on the post-reset activity of its nearest neighbours, which produces a variable-length memory structure.

We develop a graphical construction based on a dominating Poisson environment and introduce the clan of ancestors of a space-time point. The clan is the finite or infinite backward exploration of all events whose acceptance decisions may influence the target value. Our main result is a constructive subcriticality criterion: if the sure-event rate exceeds the candidate-event rate, equivalently if $\beta_*/(\beta^*-\beta_*)>1$, then the clan is almost surely finite. The proof is based on an explicit dominating branching process associated with the genealogical structure of the exploration.

The finiteness of the clan yields a measurable local construction of the stationary regime. We prove existence and uniqueness of the stationary solution by a coupling argument and obtain an exact backward--forward perfect simulation algorithm. The algorithm terminates almost surely in the subcritical regime and returns exact samples from the stationary law. Numerical experiments, together with reproducibility details and R code, illustrate the finite-clan mechanism and the computational behaviour near the theoretical threshold.
\end{abstract}

\maketitle

\bigskip
\noindent\textit{Keywords:} nonlinear Hawkes processes; interacting point processes; reset; variable length memory; clan of ancestors; perfect simulation; graphical construction; stationary process.

\smallskip
\noindent\textit{MSC 2020:} 60G55; 60J35; 60K35; 60K37; 65C05.

\section{Introduction}
\label{sec:intro}

\subsection{Motivation and framework}

Hawkes processes, introduced by Hawkes \cite{Hawkes}, are point processes whose intensity depends on past events. They provide a natural modelling framework for interacting systems with excitation or inhibition effects, with applications ranging from neuroscience to finance.
Stability and existence questions for nonlinear Hawkes processes have been extensively studied in the literature, see for instance Brémaud and Massoulié \cite{BM} and Massoulié \cite{Mas}.
In classical Hawkes models, the intensity depends on the full past through a convolution kernel, leading to dynamics that can often be embedded into a finite-dimensional Markovian representation in specific cases, for instance with exponential kernels. However, in many biologically motivated models, a key mechanism is the reset of the membrane potential after a spike. This implies that only the activity accumulated since the last spike contributes to the current state.

This reset mechanism induces a variable-length memory structure: the effective past depends on the last jump time of each component. As a consequence, the system is intrinsically non-Markovian, since the state cannot be described by a finite-dimensional sufficient statistic independent of the past. This feature motivates the use of graphical constructions and backward exploration techniques.

\subsection{From Markovian to non-Markovian perfect simulation}

The present work builds on \cite{Gon}, where the system admits a Markovian description: the full dynamics can be encoded in the current state, allowing the use of Lyapunov and Doeblin techniques. In contrast, the Hawkes model considered here is intrinsically non-Markovian. Due to the reset mechanism, the intensity depends on the activity accumulated since the last spike, leading to a random and unbounded memory length.

As a consequence, the state of the system cannot be summarized by a finite-dimensional variable, and classical Markovian tools no longer apply. In particular, neither Lyapunov arguments nor embedded Markov chain techniques can be used directly. This requires a different approach based on graphical constructions and backward exploration of dependencies.

This transition from a Markovian to a non-Markovian framework is the main source of difficulty in the present work. In particular, the backward exploration of dependencies must account for both spatial interactions and stochastic memory lengths induced by reset times.

\subsection{Graphical construction and clan of ancestors}

We develop a graphical construction based on a dominating Poisson environment, following ideas that originate in perfect simulation and interacting systems with memory, see for instance \cite{PW,CFF,GLO,GL,FGGL}. The key object is the clan of ancestors of a target space-time point, defined as the set of past events whose acceptance decisions may influence the target value.

The backward exploration proceeds by revealing events in decreasing time order and tracking unresolved dependencies. Sure events, occurring at a minimal rate, resolve dependencies by resetting sites, while candidate events may propagate dependencies to neighbouring sites. This exploration defines a random genealogical structure. The central probabilistic estimate of the paper is a domination of this genealogy by a subcritical Galton-Watson process.

\subsection{Main contributions}

The contributions of the paper are as follows.
\begin{enumerate}[label=\rm(\roman*)]
\item We introduce a graphical representation of interacting Hawkes processes with reset and variable-length memory.
\item We define the clan of ancestors and prove its almost sure finiteness under an explicit subcriticality condition.
\item The extinction proof relies on a domination by a subcritical Galton-Watson process; see Athreya and Ney \cite{AN} for background on branching processes.
\item We construct a stationary solution as a measurable function of the underlying Poisson environment and prove uniqueness by a coupling-from-the-past argument.
\item We derive a backward-forward perfect simulation algorithm that produces exact samples from the stationary distribution.
\item We include a reproducible numerical illustration with fixed simulation parameters, seed, number of replications, and R code.
\end{enumerate}

\subsection{Organisation}

Section \ref{sec:model} introduces the model and the graphical construction. Section \ref{sec:exploration} defines the backward exploration. Section \ref{sec:branching} proves extinction of the clan via a branching-process domination. Section \ref{sec:stationarity} establishes existence and uniqueness of the stationary law. Section \ref{sec:perfect} presents the perfect simulation algorithm. Section \ref{sec:properties} provides qualitative properties, and Section \ref{sec:numerics} illustrates the results numerically. Appendix \ref{app:simulation_code} reports the core R code used in the numerical section.

\section{Model and graphical environment}
\label{sec:model}

\subsection{Reset Hawkes dynamics}

Let $I=\Z$. For each $i\in\Z$, let $N^i$ be a simple point process on $\R$. We write
\[
\Delta N_t^i := N^i(\{t\})
\]
and define the last jump time before $t$ by
\[
L_t^i := \sup\{s<t:\Delta N_s^i=1\},
\]
with the convention $L_t^i=-\infty$ if the set is empty.

For a nonnegative measurable kernel $h:\R_+\to\R_+$, define the membrane potential of site $i$ at time $t$ by
\begin{equation}\label{eq:X}
X_t^i := \sum_{j:|j-i|=1}\int_{(L_t^i,t)} h(t-s)\,N^j(\dd s).
\end{equation}
The intensity of $N^i$ is assumed to be
\begin{equation}\label{eq:intensity}
\lambda_t^i = \beta(X_{t-}^i),
\end{equation}
where $\beta:\R_+\to\R_+$ is a firing rate function.

\begin{ass}\label{ass:basic}
The following assumptions hold.
\begin{enumerate}[label=\rm(\alph*)]
\item The kernel $h$ is measurable, nonnegative and integrable.
\item The function $\beta$ is measurable and non-increasing.
\item There exist constants $0<\beta_*\leq \beta^*<\infty$ such that
\[
\beta_*\leq \beta(x)\leq \beta^*,\qquad x\geq 0.
\]
\end{enumerate}
\end{ass}

The monotonicity of $\beta$ is natural for inhibitory interactions: a larger membrane potential reduces the firing intensity. The construction below mainly uses boundedness and the reset structure. Monotonicity is useful to interpret the model and to align with inhibitory neuronal systems.

We define
\begin{equation}\label{eq:delta}
\delta:=\frac{\beta_*}{\beta^*-\beta_*}.
\end{equation}

\subsection{Marked Poisson construction}

Let
\[
\Pi = \{\Pi^i:i\in\Z\}
\]
be a family of independent Poisson point processes on $\R\times\{s,c\}\times[0,1]$. A point $(t,s,u)$ of $\Pi^i$ is interpreted as a sure event at site $i$ and time $t$, while a point $(t,c,u)$ is a candidate event. The intensity measure of $\Pi^i$ is
\[
\beta_*\,\dd t\,\delta_s(\dd \tau)\,\dd u
+ (\beta^*-\beta_*)\,\dd t\,\delta_c(\dd \tau)\,\dd u.
\]
Thus sure events at each site have rate $\beta_*$ and candidate events have rate $\beta^*-\beta_*$.

Given a càdlàg adapted family of membrane potentials, the thinning rule is:
\begin{enumerate}[label=\rm(\roman*)]
\item a sure event is always accepted;
\item a candidate event at site $i$ and time $t$ with mark $u$ is accepted if
\begin{equation}\label{eq:candidate_accept}
 u\leq p(X_{t-}^i),
 \qquad
 p(x):=\frac{\beta(x)-\beta_*}{\beta^*-\beta_*}.
\end{equation}
\end{enumerate}
Since $p(x)\in[0,1]$, this rule yields the desired intensity $\beta(X_{t-}^i)$.

\begin{rem}
This marked construction is equivalent to the usual thinning from a Poisson random measure of intensity $\dd t\,\dd z$ on $\R\times\R_+$. The present formulation separates sure and candidate events explicitly, which is convenient for the backward construction.
\end{rem}

\subsection{Local finiteness}

\begin{lem}\label{lem:local_finite_environment}
For every finite set $A\subset\Z$ and every bounded interval $J\subset\R$, the number of points of $\Pi$ in $A\times J\times\{s,c\}\times[0,1]$ is almost surely finite.
\end{lem}

\begin{proof}
The total intensity of this set is $|A|\,|J|\,\beta^*<\infty$. The claim follows from the definition of a Poisson point process.
\end{proof}

\begin{lem}\label{lem:no_explosion_site}
Any process obtained by accepting a subset of the points of the environment is locally finite at each site.
\end{lem}

\begin{proof}
For every site $i$ and bounded interval $J$, the number of accepted events is bounded by the total number of points of $\Pi^i$ in $J\times\{s,c\}\times[0,1]$, which is finite almost surely by Lemma \ref{lem:local_finite_environment}.
\end{proof}

\section{Backward exploration and clan of ancestors}
\label{sec:exploration}

\subsection{Neighbourhood notation}

For a finite set $A\subset\Z$, define its nearest-neighbour boundary by
\[
\partial A := \{j\in\Z\setminus A:\exists i\in A, |j-i|=1\},
\]
and define the exploration set
\[
S(A):=A\cup\partial A.
\]
We use $A$ for currently unresolved active sites and $S(A)$ for the sites at which the next revealed event may matter.

\subsection{Definition of the exploration}

Fix a target point $(i,t)$; by translation invariance it is enough to take $(0,0)$, but the notation below is written for a general target.

\begin{defin}[Backward exploration]\label{def:backward}
Set
\[
C_0=\{i\},\qquad T_0=t,
\]
and let $\mathscr J_0=\varnothing$. Suppose $(C_n,T_n,\mathscr J_n)$ has been constructed and $C_n\neq\varnothing$. Define
\[
S_n:=S(C_n).
\]
Let $T_{n+1}$ be the largest time $r<T_n$ at which at least one point of the environment occurs at a site in $S_n$. Almost surely this point is unique; write it as
\[
(J_{n+1},T_{n+1},\tau_{n+1},U_{n+1}),
\]
where $J_{n+1}\in S_n$, $\tau_{n+1}\in\{s,c\}$ and $U_{n+1}\in[0,1]$.

Set
\[
\mathscr J_{n+1}:=\mathscr J_n\cup\{(J_{n+1},T_{n+1},\tau_{n+1},U_{n+1})\}.
\]
The active set is updated by
\begin{equation}\label{eq:update}
C_{n+1}:=
\begin{cases}
C_n\setminus\{J_{n+1}\}, & \tau_{n+1}=s,\ J_{n+1}\in C_n,\\
C_n\cup\{J_{n+1}\}, & \tau_{n+1}=c,\ J_{n+1}\in \partial C_n,\\
C_n, & \text{otherwise.}
\end{cases}
\end{equation}
The exploration stops at
\[
N_{\rm ext}:=\inf\{n\geq 0:C_n=\varnothing\}.
\]
The clan of ancestors of $(i,t)$ is the stopped random object
\[
\mathcal A(i,t):=\mathscr J_{N_{\rm ext}}
\]
when $N_{\rm ext}<\infty$.
\end{defin}

We denote by $(\mathcal G_n)_{n\geq0}$ the natural filtration of the backward exploration, namely
\[
\mathcal G_n
:=
\sigma\bigl(C_k,T_k,\mathscr J_k:\,0\leq k\leq n\bigr).
\]
Equivalently, $\mathcal G_n$ contains the active sets, the explored times and all revealed marks up to step $n$.

\begin{lem}\label{lem:well_defined_step}
For every $n\geq 0$, on $\{C_n\neq\varnothing\}$, $T_{n+1}$ is almost surely well defined and satisfies $T_{n+1}>-\infty$.
\end{lem}

\begin{proof}
The set $S_n$ is finite and non-empty. The superposition of all event processes at sites in $S_n$ is a homogeneous Poisson process on $\R$ with rate $|S_n|\beta^*$. Such a process has a last point before every deterministic time $T_n$. Since $T_n$ is measurable with respect to the previously explored environment, the same statement follows by the strong Markov property of Poisson point processes. The point is almost surely unique because the process is simple.
\end{proof}

\begin{lem}\label{lem:mark_distribution}
Conditionally on the explored past up to step $n$ and on $C_n$, the mark type of the next event satisfies
\[
\PP(\tau_{n+1}=s\mid\mathcal G_n)=p_0,
\qquad
\PP(\tau_{n+1}=c\mid\mathcal G_n)=1-p_0,
\]
where
\[
p_0:=\frac{\beta_*}{\beta^*}.
\]
Here $\mathcal G_n$ denotes the sigma-field generated by the exploration up to step $n$.
\end{lem}

\begin{proof}
At every site, sure and candidate events are independent Poisson processes with rates $\beta_*$ and $\beta^*-\beta_*$ respectively. Over the finite set $S_n$, the total sure-event rate is $|S_n|\beta_*$ and the total candidate-event rate is $|S_n|(\beta^*-\beta_*)$. Therefore the next event in $S_n$ is sure with conditional probability
\[
\frac{|S_n|\beta_*}{|S_n|\beta^*}=\frac{\beta_*}{\beta^*}=p_0.
\]
The independence from the previously explored part follows from the strong Markov property and independent increments of the Poisson environment on the unexplored intervals.
\end{proof}

\subsection{A closure property}

The next statement is deterministic once the environment and the exploration are fixed. It is the essential reason why the backward exploration is useful.

\begin{lem}[Ancestral closure]\label{lem:ancestral_closure}
Assume $N_{\rm ext}<\infty$. Let $\mathscr J:=\mathscr J_{N_{\rm ext}}$ be the revealed list. Any event of the environment that can affect the acceptance status of an event in $\mathscr J$, or the value of the target membrane potential, is itself contained in $\mathscr J$.
\end{lem}

\begin{proof}
We prove the statement by induction along the backward construction. At step $n$, the active set $C_n$ represents the set of sites whose unresolved post-reset histories may still be needed to determine the target. The enlarged set $S_n=C_n\cup\partial C_n$ contains precisely those sites at which an event can either resolve an active history or contribute presynaptically to one.

If the next event in $S_n$ is a sure event at an active site, then it is necessarily accepted and resets that site. Consequently, all earlier events affecting that site before this sure event are screened by the reset and cannot influence any later membrane potential of that site. Removing the site from $C_n$ is therefore valid.

If the next event is a candidate event at a boundary site, then its possible acceptance may contribute to the potential of a neighbouring active site. In order to decide this candidate event, one must know the post-reset history of the boundary site. The update adds this site to the active set.

If an event occurs outside $S_n$, it is neither at an active site nor at a nearest neighbour of an active site. Since interactions are nearest-neighbour, such an event cannot influence the membrane potential of any active site before the active set changes. It therefore cannot belong to any ancestral chain leading to the target unless it becomes relevant later. If it becomes relevant later, it will then be in the corresponding enlarged set and will be revealed by the exploration.

When the exploration stops, $C_{N_{\rm ext}}=\varnothing$. Thus every unresolved dependency has been resolved by a sure reset. Events before the stopping time are screened by these resets. Hence no event outside the revealed list can affect the target value or the acceptance decisions of the revealed events.
\end{proof}

\section{Subcritical extinction of the clan}
\label{sec:branching}

This section gives a domination argument. The issue is that births in the true exploration are not naturally attached to a unique parent in a Markovian way. We therefore construct an auxiliary exploration that dominates the true active set and has a branching representation.

\subsection{Dominating exploration}

For a finite active set $C$, define its boundary size by $|\partial C|$. In the nearest-neighbour one-dimensional lattice, each active site has at most two boundary neighbours, and hence
\begin{equation}\label{eq:boundary_bound}
|\partial C|\leq 2|C|.
\end{equation}
The true exploration removes an active site only when a sure event occurs at that site; it adds a new active site only when a candidate event occurs on the boundary.

We introduce a continuous-time birth-death process $Y=(Y_t)_{t\geq 0}$ on $\N$ with the following transition rates:
\[
y\to y-1 \quad \text{at rate } \beta_* y,
\]
and
\[
y\to y+1 \quad \text{at rate } 2(\beta^*-\beta_*)y.
\]
This process is a linear birth-death process. It dominates the cardinality of the clan active set in continuous backward time, because each active site can be removed at rate $\beta_*$ and can generate boundary candidate additions at total rate at most $2(\beta^*-\beta_*)$.

The previous rough domination gives the sufficient condition
\[
2(\beta^*-\beta_*)<\beta_*.
\]
However, the exploration rule used in this paper samples from $S(C)$ and counts only first entrances into the active set. A sharper discrete domination gives the condition $\beta^*-\beta_*<\beta_*$. We now state the precise version used in the sequel.

\begin{lem}[Discrete offspring domination]\label{lem:discrete_offspring_domination}
There exists a Galton-Watson process $(Z_m)_{m\geq 0}$ with $Z_0=1$ and offspring distribution
\[
\PP(\xi=k)=p_0(1-p_0)^k,
\qquad k\in\N,
\qquad p_0=\frac{\beta_*}{\beta^*},
\]
such that the total number of active-site resolutions in the clan exploration is stochastically dominated by the total progeny of $(Z_m)$.
\end{lem}

\begin{proof}
We construct a Galton-Watson tree which dominates the genealogical structure of the clan exploration.

\medskip

\noindent\textbf{Step 1: Genealogical tree.}
We first associate a rooted tree $\mathcal T$ with the clan exploration. The initial active site is declared to be the root and is assigned generation zero.

Whenever a site enters the active set for the first time, we assign it a unique parent. If the site is discovered through a boundary candidate event while a site $j$ of generation $m$ is active, then it is assigned generation $m+1$ and declared to be a child of $j$.

If several active sites could be regarded as responsible for the same discovery, we use a deterministic tie-breaking rule fixed once and for all as part of the definition of the exploration procedure, for instance the nearest site and then the smallest integer. Hence every newly discovered site has a unique parent, and $\mathcal T$ is a well-defined rooted genealogical tree.

The vertices of $\mathcal T$ are the sites that enter the active set during the exploration. The total number of active-site resolutions is bounded by the number of vertices of $\mathcal T$.

\medskip

\noindent\textbf{Step 2: Offspring control.}
Fix a vertex $j$ of $\mathcal T$, and let $\mathcal F_j$ denote the sigma-field generated by the exploration up to the time when $j$ becomes active.

From the moment $j$ becomes active until it is resolved, it remains continuously in the active set. In particular, it is not removed and reinserted later. During this time interval, consider the successive relevant opportunities associated with $j$, namely the decision times at which either:
\begin{itemize}
\item a sure event occurs at $j$, resolving its dependence and removing $j$ from the active set, or
\item the exploration continues through a boundary candidate event which may create a new active site assigned as a child of $j$.
\end{itemize}

These opportunities are defined adaptively: their timing and their nature depend on the evolution of the whole active set. However, Lemma \ref{lem:mark_distribution} applies precisely in this adaptive setting. At each such opportunity, conditionally on the full past and on the current active configuration, the probability that $j$ is resolved by a sure event is at least
\[
p_0=\frac{\beta_*}{\beta^*}.
\]
This lower bound is uniform over all histories and all active configurations.

Let $N_j$ be the number of children assigned to $j$. For $r\geq0$, the event $\{N_j\ge r+1\}$ implies that, before the resolving sure event of $j$, there have been at least $r+1$ relevant opportunities at which the resolving alternative failed. Since $j$ remains active continuously until its resolution, these failures occur at consecutive relevant opportunities preceding the resolving event.

At each such opportunity, the conditional probability of failure is at most $1-p_0$. Iterating conditional expectations gives
\[
\PP(N_j\ge r+1\mid \mathcal F_j)
\leq (1-p_0)^{r+1}.
\]
Thus, conditionally on $\mathcal F_j$, $N_j$ is stochastically dominated by a geometric random variable $\xi$ with law
\[
\PP(\xi=k)=p_0(1-p_0)^k,
\qquad k\in\N.
\]

\medskip

\noindent\textbf{Step 3: Ulam-Harris coupling.}
Let
\[
\mathcal U=\bigcup_{n\ge0}\N^n,
\]
where $\N^0=\{\varnothing\}$ consists of the empty sequence, which is the root of the Ulam-Harris tree.

On an auxiliary probability space, let $(\xi_v)_{v\in\mathcal U}$ be an independent family with common distribution
\[
\PP(\xi_v=k)=p_0(1-p_0)^k,
\qquad k\in\N.
\]

We recursively embed $\mathcal T$ into $\mathcal U$. The root of $\mathcal T$ is mapped to $\varnothing$. Suppose that a vertex of $\mathcal T$ has been mapped to $v\in\mathcal U$, and let $N_v$ denote its actual number of children.

By the conditional stochastic domination above and the standard quantile coupling, we may couple $N_v$ with $\xi_v$ so that
\[
N_v\leq \xi_v
\qquad \text{almost surely}.
\]
The $N_v$ actual children are then mapped to
\[
(v,1),\ldots,(v,N_v).
\]

This recursive construction gives an injective map from $\mathcal T$ into $\mathcal U$: children of a given vertex receive distinct indices, and vertices with different parents lie below different Ulam-Harris prefixes.

\medskip

\noindent\textbf{Step 4: Domination.}
Define
\[
\mathcal G
=
\left\{
v=(v_1,\ldots,v_n)\in\mathcal U:
v_i\leq \xi_{(v_1,\ldots,v_{i-1})}
\text{ for all }1\leq i\leq n
\right\}.
\]
Then $\mathcal G$ is the family tree of a Galton-Watson process $(Z_m)_{m\ge0}$ with $Z_0=1$ and offspring distribution
\[
\PP(\xi=k)=p_0(1-p_0)^k.
\]
The construction above ensures that the image of $\mathcal T$ is contained in $\mathcal G$. Therefore
\[
|\mathcal T|\leq |\mathcal G|
=
\sum_{m\ge0}Z_m.
\]
Since the number of active-site resolutions is bounded by $|\mathcal T|$, the desired stochastic domination follows.
\end{proof}

\begin{rem}\label{rem:condition_constant}
The domination in Lemma \ref{lem:discrete_offspring_domination} is the key place where constants enter. If one uses the simpler continuous-time cardinality domination, a stronger condition with a factor two is obtained. The discrete genealogical domination exploits the fact that only first entrances are counted as births and that every unresolved lineage is eventually killed by a sure reset. In applications where the geometry of the boundary is larger, the offspring distribution must be modified accordingly.
\end{rem}

\subsection{Extinction theorem}

\begin{theo}[Extinction of the clan]\label{theo:extinction}
Assume Assumption \ref{ass:basic}. If
\begin{equation}\label{eq:subcrit}
\delta=\frac{\beta_*}{\beta^*-\beta_*}>1,
\end{equation}
then for every target point $(i,t)$,
\[
N_{\rm ext}<\infty
\qquad\text{and}\qquad
T_{N_{\rm ext}}> -\infty
\]
almost surely.
\end{theo}

\begin{proof}
Let $(Z_m)$ be the Galton-Watson process of Lemma \ref{lem:discrete_offspring_domination}. Its offspring mean is
\[
\E[\xi]=\frac{1-p_0}{p_0}
=\frac{\beta^*-\beta_*}{\beta_*}.
\]
The condition \eqref{eq:subcrit} is exactly $\E[\xi]<1$. Hence the Galton-Watson process is subcritical and its total progeny is almost surely finite. By Lemma \ref{lem:discrete_offspring_domination}, the total number of sites ever activated in the clan is almost surely finite.

Each activated site is resolved by a sure event after finitely many revealed steps. Indeed, sure events at a fixed site form a Poisson process of positive rate $\beta_*$. Since only finitely many sites are activated, the exploration performs only finitely many activations and resolutions. Therefore $N_{\rm ext}<\infty$ almost surely.

The corresponding times satisfy
\[
t=T_0>T_1>\cdots>T_{N_{\rm ext}}.
\]
Since this sequence is finite, its last element is finite. Hence $T_{N_{\rm ext}}> -\infty$ almost surely.
\end{proof}

\begin{cor}[Exponential tail of the genealogical size]\label{cor:exp_tail}
Under \eqref{eq:subcrit}, there exist constants $C<\infty$ and $\rho\in(0,1)$ such that
\[
\PP(N_{\rm ext}\geq n)\leq C\rho^n,
\qquad n\geq 1.
\]
\end{cor}

\begin{proof}
The total progeny of a subcritical Galton-Watson process with geometric offspring distribution has an exponential moment in a neighbourhood of the origin. This follows, for example, from the fixed-point equation for the generating function of the total progeny. Since $N_{\rm ext}$ is stochastically dominated by a finite multiple of the total progeny in Lemma \ref{lem:discrete_offspring_domination}, the stated bound follows.
\end{proof}

\begin{rem}
The constants in Corollary \ref{cor:exp_tail} are not optimized. The result is used only to justify convergence in the coupling argument and finite expected simulation cost in compact subcritical regions.
\end{rem}

\section{Stationary construction and uniqueness}
\label{sec:stationarity}

\subsection{Forward evaluation of a finite clan}

Let $\mathscr J$ be the finite revealed list of the clan of a target point $(i,t)$. Sort its elements by increasing time:
\[
(r_1,j_1,\tau_1,u_1),\ldots,(r_M,j_M,\tau_M,u_M),
\qquad r_1<\cdots<r_M\leq t.
\]
For each site $k$ appearing in the list, maintain a finite set $P_k$ of accepted presynaptic events since the last accepted event of $k$.

The forward evaluation is defined recursively:
\begin{enumerate}[label=\rm(\roman*)]
\item if $\tau_m=s$, the event is accepted;
\item if $\tau_m=c$, compute
\[
x_m:=\sum_{(\ell,r)\in P_{j_m}} h(r_m-r)
\]
and accept the event if
\[
u_m\leq \frac{\beta(x_m)-\beta_*}{\beta^*-\beta_*};
\]
\item if the event is accepted, reset $P_{j_m}$ to $\varnothing$ and add $(j_m,r_m)$ to $P_{j_m-1}$ and $P_{j_m+1}$.
\end{enumerate}
At the end, define
\begin{equation}\label{eq:target_value_forward}
\Phi_{i,t}(\Pi):=\sum_{(\ell,r)\in P_i} h(t-r).
\end{equation}

\begin{lem}[Finite dependence]\label{lem:finite_dependence}
On the event that the clan of $(i,t)$ is finite, the value $\Phi_{i,t}(\Pi)$ defined by \eqref{eq:target_value_forward} is the unique value at $(i,t)$ compatible with the thinning rule and the revealed environment.
\end{lem}

\begin{proof}
By Lemma \ref{lem:ancestral_closure}, all events that can influence the target value or the acceptance status of an event influencing the target are contained in the revealed list. The forward recursion processes these events in chronological order. At every step, the set $P_k$ contains exactly the accepted presynaptic events of site $k$ occurring after its last accepted event within the revealed ancestral list. This assertion is proved by induction over the ordered list. It is true initially because no event has yet been processed. When a sure or accepted candidate event occurs at $k$, the reset of $k$ is represented by clearing $P_k$, and its contribution to neighbouring sites is represented by adding the event to their presynaptic lists. Candidate acceptance is decided using precisely the potential computed from $P_k$. Thus the recursion exactly reproduces the thinning rule on the finite ancestral list. The final expression is exactly the membrane potential at the target point.
\end{proof}

\subsection{Existence}

\begin{theo}[Existence of a stationary solution]\label{theo:existence}
Under Assumption \ref{ass:basic} and \eqref{eq:subcrit}, there exists a stationary solution to the reset Hawkes system.
\end{theo}

\begin{proof}
Let $\Pi$ be the two-sided marked Poisson environment. For every target point $(i,t)$, define
\[
X_t^i:=\Phi_{i,t}(\Pi),
\]
where $\Phi_{i,t}$ is obtained by backward exploration followed by forward evaluation. By Theorem \ref{theo:extinction}, the clan is finite almost surely, so $X_t^i$ is well defined for every fixed $(i,t)$.

The mapping $\Pi\mapsto X_t^i$ is measurable by Lemma \ref{lem:well_defined_step}, Lemma \ref{lem:ancestral_closure} and the finite forward recursion. The family is stationary under joint time shifts and spatial translations because the Poisson environment has these invariances and the exploration rule is covariant under them.

Finally define $N^i$ as the point process of accepted events at site $i$ under the same forward acceptance rule. Sure events are always accepted. Candidate events are accepted with probability determined by \eqref{eq:candidate_accept}. Therefore the stochastic intensity of $N^i$ is exactly $\beta(X_{t-}^i)$, and the system satisfies the thinning equation.
\end{proof}

\subsection{Uniqueness}

\begin{theo}[Uniqueness of the stationary law]\label{theo:uniqueness}
Under Assumption \ref{ass:basic} and \eqref{eq:subcrit}, the stationary law constructed in Theorem \ref{theo:existence} is the unique stationary law satisfying the thinning equation.
\end{theo}

\begin{proof}
Let $\nu$ be another stationary law satisfying the thinning equation. Couple a process with law $\nu$ and the canonical construction of Theorem \ref{theo:existence} by using the same Poisson environment on a finite time window $[-T,0]$.

Fix a finite set of target sites $A\subset\Z$ and a bounded measurable function $f$ depending on $(X_0^i)_{i\in A}$. Let $B_T$ be the event that the union of the clans of all target points $(i,0)$, $i\in A$, is contained in $[-T,0]$. By Theorem \ref{theo:extinction},
\[
\PP(B_T^c)\longrightarrow 0.
\]
On $B_T$, Lemma~\ref{lem:finite_dependence} implies that the vector $(X_0^i)_{i\in A}$ is a measurable function only of the common environment on $[-T,0]$. Therefore the two coupled processes agree on this vector on $B_T$. Consequently,
\[
\left|\E_\nu f((X_0^i)_{i\in A}) - \E f((\widetilde X_0^i)_{i\in A})\right|
\leq 2\|f\|_\infty\PP(B_T^c),
\]
where $\widetilde X$ denotes the canonical stationary process. Letting $T\to\infty$ gives equality of finite-dimensional distributions at time zero. By stationarity, equality holds for arbitrary finite collections of space-time points. Hence the stationary law is unique.
\end{proof}

\section{Perfect simulation}
\label{sec:perfect}

\subsection{Backward-forward algorithm}

The preceding construction gives an exact simulation algorithm for the stationary value at a target point. We state it for $(0,0)$.

\begin{algorithm}[t]
\caption{Backward exploration for the target $(0,0)$}
\label{alg:backward_v2}
\begin{algorithmic}[1]
\State $C\leftarrow\{0\}$; $t\leftarrow 0$; $J\leftarrow\emptyset$
\While{$C\neq\emptyset$}
  \State $S\leftarrow C\cup\partial C$
  \State Generate the latest Poisson event before $t$ among sites in $S$
  \State Denote this event by $(j,r,\tau,u)$
  \State Append $(j,r,\tau,u)$ to $J$
  \If{$\tau=s$ and $j\in C$}
     \State $C\leftarrow C\setminus\{j\}$
  \ElsIf{$\tau=c$ and $j\in\partial C$}
     \State $C\leftarrow C\cup\{j\}$
  \EndIf
  \State $t\leftarrow r$
\EndWhile
\State \Return $J$
\end{algorithmic}
\end{algorithm}

\begin{algorithm}[t]
\caption{Forward evaluation}
\label{alg:forward_v2}
\begin{algorithmic}[1]
\State Sort $J$ by increasing time
\State Initialise $P_k\leftarrow\emptyset$ for all sites appearing in $J$ and their neighbours
\For{each event $(j,r,\tau,u)$ in chronological order}
  \If{$\tau=s$}
     \State $A\leftarrow 1$
  \Else
     \State $x\leftarrow \sum_{(\ell,q)\in P_j} h(r-q)$
     \State $A\leftarrow \ind\left\{u\leq \frac{\beta(x)-\beta_*}{\beta^*-\beta_*}\right\}$
  \EndIf
  \If{$A=1$}
     \State $P_j\leftarrow\emptyset$
     \State $P_{j-1}\leftarrow P_{j-1}\cup\{(j,r)\}$
     \State $P_{j+1}\leftarrow P_{j+1}\cup\{(j,r)\}$
  \EndIf
\EndFor
\State \Return $\sum_{(\ell,q)\in P_0}h(-q)$
\end{algorithmic}
\end{algorithm}

\subsection{Correctness theorem}

\begin{theo}[Perfect simulation]\label{theo:perfect}
Under Assumption \ref{ass:basic} and \eqref{eq:subcrit}, Algorithms \ref{alg:backward_v2} and \ref{alg:forward_v2} terminate almost surely and return an exact sample from the stationary distribution of $X_0^0$.
\end{theo}

\begin{proof}
The backward algorithm is precisely the exploration of Definition \ref{def:backward}. Its almost-sure termination follows from Theorem \ref{theo:extinction}. Conditional on termination, the returned list is exactly the finite clan of ancestors of $(0,0)$. The forward algorithm is exactly the finite recursion described in Section \ref{sec:stationarity}. By Lemma \ref{lem:finite_dependence}, its output equals the stationary value $X_0^0$ constructed from the two-sided Poisson environment. Hence the output has exactly the stationary law.
\end{proof}

\begin{cor}[Exact joint sampling]\label{cor:joint}
For any finite set $A\subset\Z$, starting the backward exploration with $C_0=A$ and applying the same forward evaluation yields an exact sample of $(X_0^i)_{i\in A}$ under the stationary law.
\end{cor}

\begin{proof}
Start the backward exploration from the active set $C_0=A$ and reveal the union of the ancestral dependencies needed to determine all target values $(X_0^i)_{i\in A}$. By the same domination argument as in Theorem \ref{theo:extinction}, this joint clan is finite almost surely under \eqref{eq:subcrit}.

The forward evaluation is then performed once on the single revealed list, sorted in chronological order. If the ancestral clans of two target sites overlap, there is no ambiguity: the same Poisson environment and the same marks are used for the common events, and each acceptance decision is made only once from the common past reconstructed by the forward recursion. Hence overlapping parts of the clans cannot lead to contradictory decisions. Lemma \ref{lem:finite_dependence} therefore applies simultaneously to all targets in $A$, and the output has exactly the joint stationary law of $(X_0^i)_{i\in A}$.
\end{proof}

\section{Elementary structural properties}
\label{sec:properties}

\begin{prop}[Atom at zero]\label{prop:atom_zero_v2}
Under the stationary law,
\[
\PP(X_0^0=0)>0.
\]
\end{prop}

\begin{proof}
Fix $a>0$. Consider the event that site $0$ has a sure event at some time $\sigma\in(-a,0)$ and that neither neighbouring site $-1$ nor $1$ has any accepted event in $(\sigma,0)$. The first event has probability $1-e^{-\beta_*a}>0$. Conditionally on $\sigma$, accepted events at each neighbouring site are dominated by a Poisson process of rate $\beta^*$, so the probability of no accepted neighbouring event in $(\sigma,0)$ is at least $e^{-2\beta^*a}>0$. Hence the intersection has positive probability.

On this event, site $0$ has been reset at time $\sigma$ and receives no accepted presynaptic spike afterwards. Therefore $X_0^0=0$.
\end{proof}

\begin{rem}
The result shows that the reset mechanism leaves a visible trace in stationarity: the membrane potential has a non-trivial atom at zero. We deliberately do not claim monotonicity in parameters without a separate comparison theorem.
\end{rem}

\section{Numerical experiments}
\label{sec:numerics}

This section illustrates the theoretical construction and the behaviour predicted by the branching domination. The numerical experiments are not used in the proofs. They are included for three purposes only: to visualise the distribution of the dominating genealogical process, to show the behaviour of the stationary membrane potential under changes in the memory-decay parameter when the full perfect sampler is implemented, and to document the increase in computational cost near the subcriticality threshold.

\subsection{Parametric specification}

We use the following parametric family:
\begin{equation}\label{eq:numerical_beta_h}
\beta(x)=\frac{\beta^*+x\beta_*}{1+x},
\qquad
h(t)=\mu e^{-\mu t},\quad t\geq 0.
\end{equation}
The function $\beta$ is non-increasing, satisfies $\beta(0)=\beta^*$ and converges to $\beta_*$ as $x\to\infty$. The exponential kernel is used only for numerical convenience; the theoretical results do not require this specific form.

Unless otherwise stated, the baseline parameters are
\begin{equation}\label{eq:numerical_baseline}
\beta^*=3,
\qquad
\beta_*=2,
\qquad
\delta=\frac{\beta_*}{\beta^*-\beta_*}=2,
\end{equation}
so that the subcriticality condition \eqref{eq:subcrit} holds. The first and third experiments are based on the Galton-Watson process that dominates the clan in Lemma \ref{lem:discrete_offspring_domination}. They should therefore be read as numerical illustrations of the theoretical upper bound, not as direct simulations of the full Poisson backward exploration. The second experiment, concerning the stationary membrane potential, requires the full backward--forward sampler described in Algorithms~\ref{alg:backward_v2} and \ref{alg:forward_v2}.

All simulations were performed in \texttt{R}. Unless otherwise stated, each experiment is based on $M=10^4$ independent replications with random seed 12345. Confidence intervals displayed in the numerical figures are empirical 95\% intervals obtained by nonparametric bootstrap. The code reported in Appendix \ref{app:simulation_code} gives the reproducible implementation of the dominating Galton-Watson benchmark used for Figures \ref{fig:clan_size} and \ref{fig:complexity_delta}.

\subsection{Distribution of the backward exploration size}

The first quantity of interest is the total progeny of the Galton-Watson process that dominates the number of active-site resolutions in the backward exploration. This is not the full Poisson exploration itself, but the dominating object used in the proof of Theorem \ref{theo:extinction} and Corollary \ref{cor:exp_tail}.

Figure \ref{fig:clan_size} displays the empirical distribution of the total progeny of the dominating Galton-Watson process under the baseline parameters. The right panel presents the same distribution on a logarithmic scale. The purpose of the figure is not to estimate the optimal tail exponent, but to verify that the simulated exploration behaves consistently with the finite-clan regime predicted by the theory.

\begin{figure}[t]
\centering
\begin{minipage}{0.47\textwidth}
\centering
\includegraphics[width=\textwidth]{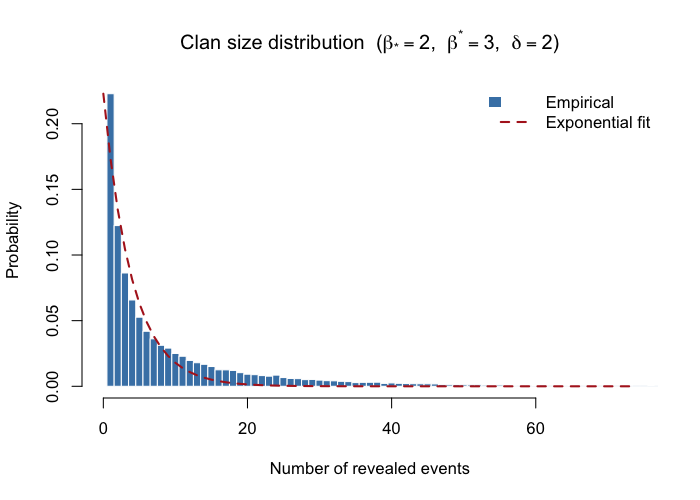}
\end{minipage}
\hfill
\begin{minipage}{0.47\textwidth}
\centering
\includegraphics[width=\textwidth]{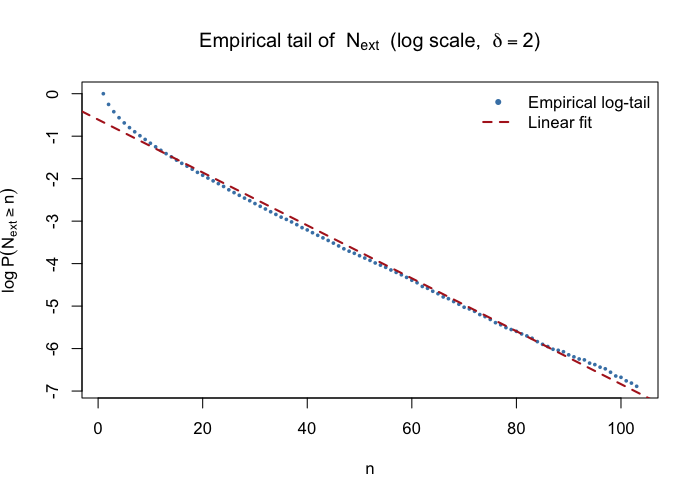}
\end{minipage}
\caption{Empirical distribution of the total progeny of the dominating Galton-Watson process. Left: histogram. Right: empirical tail on a logarithmic scale. The behaviour is consistent with the exponential-tail estimate of Corollary \ref{cor:exp_tail}.}
\label{fig:clan_size}
\end{figure}

\subsection{Stationary membrane potential and memory decay}

We next consider the stationary membrane potential $X_0^0$ obtained by the exact backward-forward algorithm. For the exponential kernel in \eqref{eq:numerical_beta_h}, the parameter $\mu$ controls how fast past presynaptic spikes are discounted. Larger values of $\mu$ correspond to shorter memory.

Figure \ref{fig:stationary_potential} shows empirical distribution functions of $X_0^0$ for several values of $\mu$ in the subcritical regime. The curves illustrate the intuitive effect of the memory-decay parameter: when $\mu$ is larger, older presynaptic inputs are attenuated more rapidly, and the simulated distribution places more mass near zero. This observation is presented as an illustration only. We do not claim a stochastic monotonicity theorem in $\mu$.

\begin{figure}[t]
\centering
\includegraphics[width=0.70\textwidth]{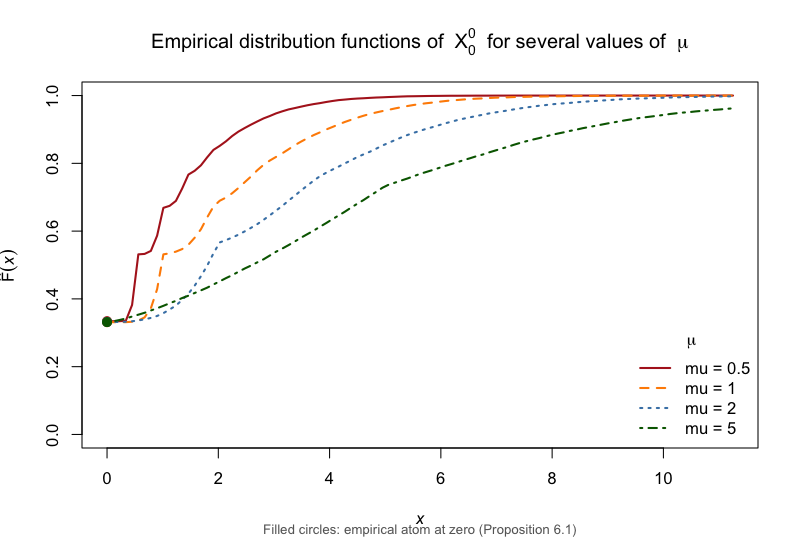}
\caption{Empirical distribution functions of the stationary membrane potential $X_0^0$ for different values of the exponential decay parameter $\mu$. The figure illustrates the effect of memory decay in the exact stationary samples. No stochastic monotonicity statement is asserted.}
\label{fig:stationary_potential}
\end{figure}

\subsection{Computational cost near the threshold}

Finally, we examine the average total progeny of the dominating Galton-Watson process as a function of
\[
\delta=\frac{\beta_*}{\beta^*-\beta_*}.
\]
The theory ensures almost-sure termination when $\delta>1$. As $\delta$ approaches the threshold from above, the mean offspring in the dominating branching process approaches one, and the backward exploration is expected to become more costly.

Figure \ref{fig:complexity_delta} reports the empirical mean total progeny of the dominating process for values of $\delta$ larger than one. The plot is consistent with the branching-process interpretation of the proof: the closer the parameter is to the threshold, the larger the observed exploration size. This figure documents the practical computational behaviour of the exact sampler in the regime where the theoretical condition applies.

\begin{figure}[t]
\centering
\includegraphics[width=0.70\textwidth]{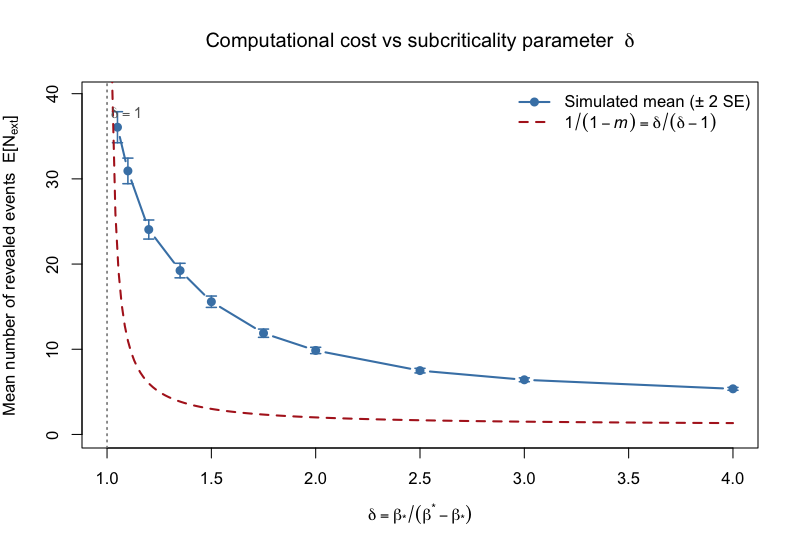}
\caption{Empirical mean total progeny of the dominating Galton-Watson process as a function of $\delta=\beta_*/(\beta^*-\beta_*)$, restricted to the subcritical side $\delta>1$. The vertical dashed line marks the theoretical threshold. The increase near the threshold is consistent with the branching-process domination used in the proof of Theorem \ref{theo:extinction}.}
\label{fig:complexity_delta}
\end{figure}

\begin{rem}
The numerical section is deliberately limited to three figures. The first one illustrates the finite-clan mechanism, the second one illustrates the stationary output of the exact sampler, and the third one documents the computational cost near the theoretical threshold.
\end{rem}

\section{Conclusion}

We developed a graphical construction and a perfect simulation algorithm for interacting Hawkes processes with reset-induced variable-length memory. The central object of the analysis is the clan of ancestors, which provides a finite representation of the dependencies determining a given space-time point under an explicit subcriticality condition.

The present work extends earlier perfect simulation constructions for interacting neuronal systems developed in \cite{Gon}. While those models admit a Markovian formulation, the Hawkes dynamics considered here lead to an intrinsically non-Markovian system, in which the relevant past is determined by random reset times. This transition fundamentally modifies the analytical framework: the state of the system is no longer finite-dimensional, and the backward exploration must handle stochastic memory horizons.

The extinction of the clan relies on a domination by a subcritical branching process, a classical mechanism in the study of interacting systems with random genealogies; see Athreya and Ney \cite{AN}. This provides a constructive route to both existence and uniqueness of the stationary regime, as well as an exact simulation procedure.

Several directions remain open. A first question is to refine the subcriticality condition and identify optimal thresholds. A second direction concerns quantitative properties of the stationary regime, such as mixing rates or concentration inequalities. Finally, extending the approach to more general interaction graphs or non-homogeneous settings would broaden the applicability of the method.

\bigskip
\paragraph{Acknowledgements.}
The authors thank Eva L\"ocherbach for insightful discussions on Hawkes processes and interacting systems with memory.

\appendix

\section{Simulation details and R code}
\label{app:simulation_code}

This appendix reports the core R implementation used for the Galton-Watson benchmark in Figures \ref{fig:clan_size} and \ref{fig:complexity_delta}. This code does not simulate the full Poisson backward exploration of Algorithm \ref{alg:backward_v2}; it simulates the dominating branching process appearing in Lemma \ref{lem:discrete_offspring_domination}. The distinction is important: the benchmark illustrates the theoretical upper bound on the clan genealogy, whereas the exact sampler is the backward-forward algorithm described in Section \ref{sec:perfect}.

\begin{verbatim}
simulate_dominating_gw <- function(p0, max_iter = 100000) {
  active <- c(0)
  total_progeny <- 1
  revealed <- 0

  while (length(active) > 0 && revealed < max_iter) {
    j <- active[1]
    active <- active[-1]

    repeat {
      revealed <- revealed + 1
      if (runif(1) < p0) {
        break
      } else {
        total_progeny <- total_progeny + 1
        active <- c(active, total_progeny)
      }
      if (revealed >= max_iter) {
        warning("Maximum number of iterations reached.")
        break
      }
    }
  }

  return(list(
    total_progeny = total_progeny,
    revealed_events = revealed
  ))
}

set.seed(12345)
M <- 10000
beta_star <- 2
beta_upper <- 3
p0 <- beta_star / beta_upper

sim <- replicate(M, simulate_dominating_gw(p0), simplify = FALSE)

total_progeny  <- sapply(sim, function(x) x$total_progeny)
revealed_events <- sapply(sim, function(x) x$revealed_events)

summary(total_progeny)
summary(revealed_events)

# Bootstrap 95\% sur la moyenne de total_progeny
B <- 1000
boot_means <- replicate(B, mean(sample(total_progeny, M, replace = TRUE)))
ci_lower <- quantile(boot_means, 0.025)
ci_upper <- quantile(boot_means, 0.975)
cat("IC bootstrap 95\% sur la moyenne :",
    round(ci_lower, 4), "-", round(ci_upper, 4), "\n")
\end{verbatim}

For the sensitivity analysis in Figure \ref{fig:complexity_delta}, the same routine is repeated over a grid of values of
\[
\delta=\frac{\beta_*}{\beta^*-\beta_*}>1.
\]
For each value of $\delta$, the empirical mean total progeny is computed over $10^4$ independent replications.

The numerical experiments are illustrative only. They are not used in the proofs of the theoretical results.


\begin{thebibliography}{99}

\bibitem{AN}
Athreya, K.\,B. and Ney, P.\,E. (1972).
\textit{Branching Processes}.
Springer, Berlin.

\bibitem{BM}
P. Br\'emaud and L. Massouli\'e.
\newblock Stability of nonlinear Hawkes processes.
\newblock \emph{Ann. Probab.} \textbf{24} (1996), 1563--1588.

\bibitem{CFF}
F. Comets, R. Fernandez and P.A. Ferrari.
\newblock Processes with long memory: regenerative construction and perfect simulation.
\newblock \emph{Ann. Appl. Probab.} \textbf{12} (2002), 921--943.

\bibitem{DFH}
S. Delattre, N. Fournier and M. Hoffmann.
\newblock Hawkes processes on large networks.
\newblock \emph{Ann. Appl. Probab.} \textbf{26} (2016), 216--261.

\bibitem{FGGL}
P. A. Ferrari, A. Galves, I. Grigorescu and E. L\"ocherbach.
\newblock Phase transition for infinite systems of spiking neurons.
\newblock \emph{J. Stat. Phys.} \textbf{172} (2018), 1564--1575.

\bibitem{GL}
A. Galves and E. L\"ocherbach.
\newblock Infinite systems of interacting chains with memory of variable length.
\newblock \emph{J. Stat. Phys.} \textbf{151} (2013), 896--921.

\bibitem{GLO}
A. Galves, E. L\"ocherbach and E. Orlandi.
\newblock Perfect simulation of infinite range Gibbs measures and coupling with their finite range approximations.
\newblock \emph{J. Stat. Phys.} \textbf{138} (2010), 1077--1103.

\bibitem{Gon}
Goncalves, B.\,P.\,I. (2023).
An interacting neuronal network with inhibition: theoretical analysis and perfect simulation.
\textit{MathematicS In Action}, Maths Bio, \textbf{12}(1), 3--22.
\texttt{doi:10.5802/msia.29}.

\bibitem{Hawkes}
A. G. Hawkes.
\newblock Point spectra of some mutually exciting point processes.
\newblock \emph{J. Roy. Statist. Soc. Ser. B} \textbf{33} (1971), 438--443.

\bibitem{HL}
P. Hodara and E. L\"ocherbach.
\newblock Hawkes processes with variable length memory and an infinite number of components.
\newblock \emph{Adv. in Appl. Probab.} \textbf{49} (2017), 84--107.

\bibitem{Mas}
L. Massouli\'e.
\newblock Stability results for a general class of interacting point processes dynamics.
\newblock \emph{Stochastic Process. Appl.} \textbf{75} (1998), 1--30.

\bibitem{PW}
J. G. Propp and D. B. Wilson.
\newblock Exact sampling with coupled Markov chains and applications to statistical mechanics.
\newblock \emph{Random Structures Algorithms} \textbf{9} (1996), 223--252.

\end{thebibliography}
\end{document}